\documentclass[11pt, oneside]{amsart}
\calclayout
\usepackage{hyperref} 
\usepackage{amsmath}
\usepackage{amssymb}
\usepackage{amsthm} 
\usepackage{fullpage}  
\usepackage{cleveref} 
\usepackage{blindtext}
\usepackage{graphicx}
\usepackage{pgf,tikz}
\usepackage{mathrsfs}
\usepackage{comment}
\usetikzlibrary{arrows}
\usepackage{MnSymbol}
\usepackage{MnSymbol,wasysym}
\usepackage{array}
\usepackage{mathtools} 
\usepackage{graphicx}

\usepackage[english]{babel}
\usepackage[utf8]{inputenc}
\usepackage{amssymb}
\usepackage{pb-diagram}
\usepackage{mathtools}
\usepackage{mathrsfs} 
\usepackage{array}
\usepackage{tabularx}
\usepackage{pgfkeys}
\usepackage{csquotes}
\usepackage{booktabs} 
\usepackage{float} 
\usepackage{amsthm}
\usepackage{amsmath}
\numberwithin{equation}{section}
\newtheorem{definition}[equation]{Definition}
\newtheorem{theorem}[equation]{Theorem}
\newtheorem*{theorem*}{Theorem}

\newtheorem{lemma}[equation]{Lemma}
\newtheorem{proposition}[equation]{Proposition}
\newtheorem{corollary}[equation]{Corollary}
\theoremstyle{definition}
\newtheorem{example}[equation]{Example}
\theoremstyle{remark} 
\newtheorem{remark}[equation]{Remark}
\mdseries
\newcommand{\A}{\mathcal{A}}
\usepackage[
backend=bibtex,
style=alphabetic, maxnames=5
]{biblatex}
\begin{document} 
	\title{Separator-based derivations of graphic arrangements}
	\author{Leonie Mühlherr} 
	\address{Leonie Mühlherr, Universit\"at Bielefeld, Fakult\"at f\"ur Mathematik, Bielefeld, Germany}
	\email{lmuehlherr@math.uni-bielefeld.de}
	\begin{abstract}The class of subarrangements of the well-studied braid arrangement, so-called graphic hyperplane arrangements, is important for analysing new concepts and properties in hyperplane arrangement theory. While there is a nice characterization of free graphic arrangements, many interesting questions beyond the free case remain open. This paper introduces an explicit set of generators for the module of logarithmic derivations of a general graphic arrangement based on graph separators. We obtain new insights in the derivation degree sequence in the non-free case and give bounds on the highest degree in the sequence. Moreover, this can broadly be used for future research in this area.
		
	\end{abstract} 
	\maketitle
	\section{Introduction}
	A hyperplane arrangement $\A$ in an $\ell$-dimensional vector space $W$ over a field $\mathbb{K}$ is a finite collection of subspaces of dimension $\ell-1$. It can be studied through various lenses, e.g., its combinatorial structure given by the intersection lattice $L(\A)$ of the hyperplanes or its algebraic side by analysing the module of logarithmic derivations (or $\A$-derivations) $D(\A)$. The arrangement $\A$ is called free if $D(\A)$ is free as a module. A long-standing open question in this research area is Terao's conjecture. It states that over a fixed field $\mathbb{K}$ the freeness of an arrangement is only dependent on its intersection lattice $L(\A)$, i.e., its combinatorial structure. In this work, we are particularly interested in hyperplane arrangements arising from simple graphs: 
	
	\begin{definition} \label{def:graph_arr} 
		Let $\mathbb{K}$ be a field and let $W = \mathbb{K}^\ell$. Let $x_1,\dots,x_\ell$ be a basis for the dual space $W^*$. Given a simple, undirected graph $G = (V,E)$ with vertex set $V = \{1,\dots, \ell\}$, define the \emph{graphic arrangement $\mathcal{A}(G)$} by \[\mathcal{A}(G) = \{\ker(x_i-x_j) ~\vert~ \{i,j\}\in E\}.\]
	\end{definition}  
	In this setting results from graph theory are applicable to derive properties of the arrangements. The graphic arrangement corresponding to the complete graph $K_\ell$ on $\ell$ vertices is the well-studied braid arrangement.
	Thus, graphic arrangements parametrize all subarrangements of the braid arrangement. 
	Moreover, the braid arrangement arises in physics in the context of the moduli space $\mathcal{M}_{0,n}$, which can be modelled as its complement space \cite{early}.

	The freeness and related properties of graphic arrangements are well-researched. 
	A result by Stanley, Edelman and Reiner equates the class of chordal graphs (graphs which do not have an induced cycle of length greater than three) with the class of free graphic hyperplane arrangements  \autocite{edelman}. More recently, Tran and Tsujie proved that the class of strongly chordal graphs corresponds to the class of MAT-free graphic arrangements \autocite{tran2022matfree}. In previous work with Abe, Kühne and Mücksch, we proved that the class of weakly chordal graphs corresponds to graphic arrangements with projective dimension at most 1, the latter being defined again as the projective dimension of $D(\A)$  \autocite{abe2023projective}. The proof of this result partially relied on finding explicit generators for $D(\A(\overline{C_\ell}))$ for $\overline{C_\ell}$ being the $\ell$-antihole, i.e. the complement of the cycle graph $C_\ell$. \par 
	
	This present work initiates the study of an explicit and systematic way of describing the generators of $D(\A(G))$ for arbitrary graphs. The starting point is the following result for the complete graph $K_\ell$, which goes back to Saito \cite{saito}. Denote by $\partial_{x_i}, 1 \leq i \leq \ell$ the partial derivatives. 
	\begin{theorem*}[Saito]\label{thm:saito_braid} 
		Define for $k \in \mathbb{N}_0$ the derivation $\theta_k = \sum_{i =1}^\ell x_i^k \cdot \partial_{x_i}$, then $\theta_0,\dots, \theta_{\ell-1}$ is a basis of the module $D(\A(K_\ell))$. 
	\end{theorem*} 
	
	The structure of the generators of the antihole is closely related to the connectivity (see Definition \ref{def:con}) of the antihole graphs and their separator structure. A separator of a connected graph $G$ is a set of vertices $T$, such that $G\setminus T$ is disconnected (see Definition \ref{def:sep}). In this work, we always assume that the graph $G$ to be connected. 
	With a focus on this, we can generalise the concept generating systems for antihole graphs to all graphs and thus write down the first systematic method of finding a generating set.\par

	To this end, we introduce separator-based $\A$-derivations as the central objects of this work: 
	\begin{definition} \label{def:sep_based_der} 
		Let $G$ be a graph and $T$ a separator of $G$ and let $C$ be a connected component of $G\setminus T$. We define a derivation of $D(\A(G))$  by \[\theta_{C}^T = \sum_{i \in C} \prod_{t \in T} (x_i-x_t) \cdot  D_i, \]
		where $D_i$ is the partial derivative $\partial_{x_i}$.  
	\end{definition} 
	With this definition and the derivations $\theta_i$, we are able to state the main theorem: 
	
	\begin{theorem}\label{thm:main}  
		Let $G$ be a graph, then there exists a minimal generating set for $D(\A(G))$ which consists of $\theta_0,\dots, \theta_\kappa$ for $\kappa$ the connectivity of $G$ and a finite set of separator-based $\A$-derivations. 
	\end{theorem} 
	
	In a recent work, Burity and Toh\v{a}neanu analyse the minimal degree of a Jacobian relation 
	for a graphic arrangement $\A(G)$ 
	(\autocite{burity}, Appendix A). For this, they also use the (vertex)-connectivity of $G$ and its vertex-cut sets (see Definitions \ref{def:sep} and \ref{def:con}). A connection between these graph theoretical properties and the module of logarithmic derivations has thus already been established in the literature.
	
	This approach through graph theory offers new insights into classic questions surrounding minimal generating sets. For instance, Theorem \ref{thm:main} was used in a recent project on likelihood correspondence, where the defining polynomials of hyperplanes are interpreted as probability functions in a statistical model \cite{Kahle}. 
	
	Wakefield analysed the derivation degree sequence of a graphic arrangements and gives some bounds for the highest degree (\autocite{wakefield}, Corollary 4.6).  
	
	\begin{definition}\label{def:der_deg_seq} 
		Let $\mathcal{A}$ be an arrangement and $B$ be a minimal generating set of $D(\mathcal{A})$. Call the ordered sequence of polynomial degrees of all elements in $B$ the \emph{derivation degree sequence} of $\mathcal{A}$.  
	\end{definition}

	If $\mathcal{A}$ is a free arrangement, we call the derivation degree sequence the \emph{exponents} of $\mathcal{A}$. In the free case, there is a lot of research pertaining to interpreting the exponents in the different arrangement theoretical contexts. For instance, we can encode the intersection lattice of $\A$ in the so-called Poincaré-polynomial and Teraos's Factorization Theorem states that if $\A$ is free, this polynomials factors over the integers with roots exactly the exponents. In the non-free case not as much is known, yet the sequence might be used to understand the interplay of the combinatorics of $\A$ with its geometry. 
	
	With this new explicit construction, we are able to explain the findings of Wakefield in a broader framework and improve on the bounds he presents.
	
	\begin{proposition}
		Let $G$ be a graph, $\Theta_G$ a minimal generating set of $D(\A(G))$ and $d$ the maximal degree of the derivation degree sequence. Let further $\Delta(G)$ be the highest degree of all vertices in $G$, $c$ be the cardinality of the largest clique in $G$ and $t_{\max}$ the highest cardinality of a minimal separator of $G$. Then, \[\max\{c-1, t_{\max}\} \leq d \leq \Delta(G).\] 
	\end{proposition}
	The lower bound on the maximum degree in \cite{wakefield} is less or equal to $t_{\max}$. 
	
	\medskip
	Some further applications involve more specific classes of freeness and whether they have a graphic classification. For instance, it is an open question if MAT-free arrangements are always flag-accurate, a notion introduced recently by Mücksch and Röhrle \cite{mücksch}. This will be explored in future work. 
	
	This paper is organized as follows: In Section \ref{graph_theory_preliminaries} we recall the basic concepts of graph and separator theory and in Section \ref{hyperplane_arrangements_preliminaries} we recall definitions and tools in the study of freeness of hyperplane arrangements. 
	The role that connectivity of the underlying graph plays in the structure of $D(\A(G))$ in terms of similarity to the module of the braid arrangement is explored in Section \ref{construction}. Furthermore, the construction of a generating set for an arbitrary graph $G$, using derivations of the form of Definition~\ref{def:sep_based_der} is described and Theorem \ref{thm:main} is proved. We apply this construction to study the derivation degree sequence of the module of logarithmic derivations and present our findings related to the highest degree (lower and upper bounds) and subsequences in Section \ref{minimality}. 
	
	\subsection*{Acknowledgments}
	Parts of this manuscript were written during a research stay with Takuro Abe at Rikkyo University. The author would like to thank him for fruitful discussions as well as the Bielefeld Graduate School of Theoretical Sciences whose Mobility Grant supported the research stay. The author would like to thank Max Wakefield for his helpful remarks and interesting questions about the derivation degree sequence and Lukas Kühne for his supervision and guidance throughout the research process. 
	
	The author is funded by the Deutsche Forschungsgemeinschaft (DFG, German Research Foundation) –
	Project-ID 491392403 – TRR 358.

	\section{Preliminaries} 
	\subsection{Graph Theory}\label{graph_theory_preliminaries} 
	In this section, we define properties of interest to us while studying graphic arrangements.  We mostly base this section on \cite{Diestel}. 
	We only consider finite, simple graphs:

	\begin{definition}
		A simple graph $G$ on a finite set $V$ is a tuple $(V,E)$, such that $E \subseteq \binom{V}{2}$. $E$ is the set of edges connecting the vertices in $V$. 
	\end{definition}
	For a graph $G$ we write the set of vertices and edges of $G$ as $V(G), E(G)$ respectively and write $uv$ for the undirected edge $\{u,v\}$. We are interested in special types of graphs:
	\begin{definition} 
		\begin{enumerate} 
			\item A path is a non-empty graph $P = (V,E)$ of the form $V = \{x_0,\dots,x_k\}$, $E = \{x_0x_1, x_1x_2,\dots,x_{k-1}x_k\}$, where  $x_i$ are all distinct. 
			\item Given sets $A,B \subseteq V(G)$, we call $P = x_0\dots x_k$ an $A-B$ path if $V(P)\cap A = \{x_0\}$ and $V(P)\cap B = \{x_k\}$. 
			\item If $P = x_0 \dots x_{k-1} = (V,E)$ is a path, and $k \geq 3$, then the graph $C = (V, E\cup \{x_{k-1}x_0\})$ is called a ($k$-)cycle.
		\end{enumerate} 
	\end{definition} 
	
	Denote by $N(v) = \{w\in G~\vert~w~\text{is adjacent to}~v \}$ the \emph{neighbourhood} of a vertex $v$ and by $\deg(v) := \vert N(v)\vert$ its \emph{degree}. For a graph $G$ denote by $\delta(G)$ the minimal and by $\Delta(G)$ the maximal degree of the vertices of $G$. 
	
	For sets $A,B$ in $V(G)$ we say that $T\subseteq V(G)$ separates $A$ and $B$ if every $A-B$-path in $G$ contains a vertex in $T$. $T$ separates two vertices $a,b$ if it separates the sets $\{a\}, \{b\}$, but $a,b \notin T$ and we say that $T$ separates $G$ if it separates some two vertices $a,b \in V(G)$. In this case, $T$ is called an \emph{(a,b)-separator} or \emph{separator}.

	\begin{definition} \label{def:sep} 
		A set $T\subseteq V(G)$ is called a minimal $(a,b)$-separator if it is an $(a,b)$-separator and no proper subset of $T$ separates $a$ and $b$.
		
		$T$ is called a minimal separator of $G$ if it is a minimal $(a,b)$-separator for some vertices $a,b$. 
	\end{definition} 
	Denote the set of all minimal separators of a graph $G$ by $\mathscr{S}(G)$. 
	For a separator $T$ of a graph $G$, write the set of connected components of $G\setminus T$ as $\mathscr{C}_G(T)$ or just $\mathscr{C}(T)$. 
	
	Denote by $\mathscr{S}_m = \{T \in \mathscr{S}(G), \vert T \vert = m\}$. 
	Define $t_{\max}$, respectively $t_{\min}$ as  maximum, respectively minimum $m$, such that   $\mathscr{S}_m$ is non-empty. 
	
	\begin{remark} Note that the minimality definition does not state anything about whether a given $(a,b)$-separator has the minimal cardinality necessary to separate $a$ and $b$. 
	\end{remark} 
	\begin{figure}[H]
		\includegraphics[width=0.3\textwidth]{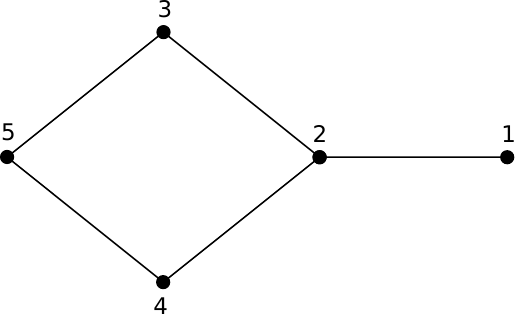} 
		\caption{Example graph with minimal separators of different sizes.} 
		\label{fig:ex:separators}
	\end{figure}
	In the example graph of Figure \ref{fig:ex:separators}, the vertex sets $\{3,4\}$ and $\{2\}$ are both minimal $(1,5)$-separators. The vertex set $\{2,5\}$ is a $(1,4)$-separator, but not minimal, since $\{2\}$ is also separating $1$ and $4$. Nevertheless, $\{2,5\}$ is a minimal $(3,4)$-separator. 
	
	\medskip
	For a vertex set $X$, denote by $G\setminus X$ the graph obtained by removing $X$ from $V(G)$ and removing all edges containing vertices in $X$. 
	Closely connected to the notion of separators is the connectivity of a graph: 
	
	\begin{definition} \label{def:con} 
		A graph $G$ is called $k$-connected, $k\in \mathbb{N}$ if $\vert V(G) \vert > k$ and $G\backslash X$ is connected for every $X \subseteq V(G)$ with $\vert X\vert < k$. 
	\end{definition}
	Stated differently, no two vertices in a $k$-connected graph are separated by fewer than $k$ other vertices. A $1$-connected graph is usually just referred to as $\emph{connected}$. The greatest integer, such that $G$ is $k$-connected is called the \emph{connectivity} of $G$, denoted $\kappa(G)$. For instance, the connectivity of the complete graph $K_\ell$ is $\ell-1$. 
	
	\subsection{Hyperplane Arrangements}\label{hyperplane_arrangements_preliminaries}  
	This section gives an overview over the essential objects and tools necessary the construction of the generator system described above. It closely follows \cite{OrlikTerao}.
	\begin{definition} 
		Let $\mathbb{K}$ be a field and let $W$  be a vector space of dimension $\ell$. A hyperplane $H$ in $W$ is an subspace of dimension $\ell-1$. A hyperplane arrangement $\mathcal{A}$  is a finite set of hyperplanes in $W$. 
	\end{definition} 
	
	Let $W^*$ be the dual space of $W$ and $S = S(W^*)$ be the symmetric algebra of $W^*$. Identify $S$ with the polynomial algebra $S = \mathbb{K}[x_1,\dots,x_\ell]$. 
	\begin{definition} 
		Let $\mathcal{A}$ be a hyperplane arrangement. Each hyperplane $H \in \mathcal{A}$ is the kernel of a polynomial $\alpha_H$ of degree 1 defined up to a constant factor. The product \[Q(\mathcal{A}) = \prod_{H \in \mathcal{A}} \alpha_H\] is called a defining polynomial of $\mathcal{A}$. 
	\end{definition} 
	Since $\alpha_H$ is a linear form for all $H$, $Q(\mathcal{A})$ is a homogeneous polynomial. 
	
	\begin{example} \label{ex:braid} 
		The braid arrangement $\A_{\ell-1}$ is the hyperplane arrangement with defining polynomial \[Q(\A_{\ell-1}) = \prod_{1\leq i < j \leq \ell} (x_i-x_j).\]
		In other words, it is the graphic arrangement associated to the complete graph $K_\ell$ (see Definition~\ref{def:graph_arr}). 
		
	\end{example} 
	
		We can associate a module to the hyperplane arrangement $\mathcal{A}$: 
		\begin{definition} 
			A $\mathbb{K}$-linear map $\theta: S \rightarrow S$ is a derivation if for $f,g \in S$ it holds that: \[\theta(f\cdot g) = f\cdot \theta(g)+g\cdot \theta(f). \] Let $\emph{Der}_{\mathbb{K}}(S)$ be the $S$-module of derivations of $S$. \\ Define an $S$-submodule of $\emph{Der}_{\mathbb{K}}(S)$, called the module of $\mathcal{A}$-derivations, by \[D(\mathcal{A}) = \{\theta \in \emph{Der}_{\mathbb{K}}(S)~\vert ~ \theta(\alpha_H)\in \alpha_H\cdot S~\text{for all}~H\in \mathcal{A}\}.\]
			The arrangement $\mathcal{A}$ is called free if $D(\mathcal{A})$ is a free $S$-module. 
		\end{definition} 
		Let $D_i$ denote the usual partial derivation $\partial_{x_i}$. One can show  that $\{D_i\}_{0\leq i \leq \ell}$ is a basis of $\text{Der}_{\mathbb{K}}(S)$ over $S$, i.e. if $\theta\in \text{Der}_{\mathbb{K}}(S)$, then $\theta = \sum \theta(x_i)\cdot D_i$ . If $D(\mathcal{A})$ is a free module, there is a criterion to determine whether a set of derivations is a basis for that module: 
		\begin{definition} \label{def:matcoef} 
			Given derivations $\nu_1,\dots,\nu_\ell\in D(\mathcal{A})$, define the coefficient matrix $M(\nu_1,\dots,\nu_\ell)$ as $M_{i,j} = \nu_j(x_i)$.
		\end{definition} 
		\begin{theorem}(Saito's criterion) \label{thm:saito}

			Given $\nu_1,\dots,\nu_\ell\in D(\mathcal{A})$, the following two conditions are equivalent: 
			\begin{enumerate} 
				\item $\det M(\nu_1,\dots,\nu_\ell) = c\cdot Q(\mathcal{A})$ for some $c\in \mathbb{K}^*$. 
				\item $\nu_1,\dots,\nu_\ell$ form a basis for $D(\mathcal{A})$ over  $S$. 
			\end{enumerate} 
		\end{theorem} 
		We will usually write homogeneous derivations in the basis $\{D_i\}_{0\leq i\leq \ell}$ or as the matrix $M(\nu_1,\dots, \nu_\ell)^t$. 
		
		Denote by $S_p$ the $\mathbb{K}$-vector space of $S$ consisting of 0 and the homogeneous polynomials of degree $p$ for $p \geq 0$. For $p < 0$, set $S_p = 0$. Then, \[S = \bigoplus_{p \in \mathbb{Z}} S_p\] is a graded $\mathbb{K}-algebra$. 
		\begin{definition} 
			A non-zero element $\theta\in \emph{Der}_{\mathbb{K}}(S)$ is homogeneous of polynomial degree p if $\theta = \sum_{k=1}^\ell f_k\cdot D_k$ and $f_k\in S_p$ for $1\leq k\leq \ell$. Write $\emph{\text{pdeg}}(\theta) = p$.
		\end{definition} 
		Let $\text{Der}_\mathbb{K}(S)_p$ be the vector space consisting of all homogeneous derivations of polynomial degree $p$ for $p\geq 0$ with $\text{Der}_\mathbb{K}(S)_p = 0$ if $p <0$. Thus, $\text{Der}_\mathbb{K}(S)$ is a graded $S$-module \[\text{Der}_\mathbb{K}(S) = \bigoplus_{p \in \mathbb{Z}} \text{Der}_\mathbb{K}(S)_p. \]
		
		\section{Constructing a Generating Set} \label{construction} 
		
		In order to construct a generating set, we need the following lemma about ideals generated by monomials. 
		
		Let $\mathcal{Z}$ be a finite set of variables and $Z_1,\dots, Z_n \subseteq \mathcal{Z}$. Define $I_j := \langle Z_j \rangle$ to be the ideal generated by the variables $Z_i$ in the polynomial ring $\mathbb{K}[\mathcal{Z}]$.  Define the set of monomials \[Z_{\text{trans}} = \left\{\prod_{z \in Y} z ~\vert ~ Y \cap Z_i \not= \emptyset ~\text{for all}~1\leq i \leq n \right\}.\]
		
		To prove the following Lemma, we use a similar inductive argument as in the proof of Theorem 2.2 of \autocite{SahaSenguptaTripathi}.
		\begin{lemma} \label{lem: mongen} 
			$Z_{\emph{trans}}$ generates the ideal $I = \bigcap_{i = 1}^n I_i$. 
		\end{lemma}
		\begin{proof} 
			Let $f = \sum_{j \in J} \lambda_j \mathbf{z}_j \in I$. Each $\mathbf{z_j}$ is a monomial and by definition  in every ideal. Thus, for each $1\leq i \leq n$, there exists $z_i \in Z_i$ with $z_i ~\vert~ \mathbf{z_j}$. The $z_i$ need not be distinct, so consider the reduced product of distinct variables $p~\vert~\prod_{i = 1}^n z_i$. It holds that $p~\vert~\mathbf{z_j}$ and $p \in Z_{\text{trans}}$ which finishes the proof. 
		\end{proof} 
		In the following, we label the vertex of a graph $G$ corresponding to the variable $x_i$ as $i$. 
		\begin{proposition} \label{connectedprop} 
			
			Let $G$ be a $k$-connected graph on $\ell$ vertices, then 
			\[\bigoplus_{j < k} D(\A(G))_j = \bigoplus_{j < k} D(\A(K_\ell))_j. \]
			That is, the derivations of polynomial degree smaller than $k$ of $D(\A(G))$ coincide with the ones from the braid arrangement. 
		\end{proposition} 
		\begin{proof} 
			
			The inclusion \[\bigoplus_{j <k} D(\A(G))_j \supseteq \bigoplus_{j <k} D(\A(K_\ell))_j\] is clear. It remains to show that for $\theta \in \bigoplus_{j < k} D(\A(G))_j$ and two vertices $v,w$, such that $\{v,w\} \notin E(G)$ it holds that $\theta(x_v-x_w) \in (x_v-x_w)\cdot S$. 
			So let $\theta \in \text{Der}_\mathbb{K}(S)_i$ for $0 \leq i < k$ and $\{v,w\} \notin E(G)$.

			\medskip 
			
			Since $G$ is $k$-connected, there exist at least $k$ disjoint paths in $G$ without common vertices \[P_1 = (v_1^1\dots v_{n_1}^1),\dots,P_{k}= (v_1^{k} \dots v_{n_{k}}^{k})\] between $v$ and $w$. 
			
			For each path $(v_1^s\dots v_{n_s}^s)$ we can write \begin{align*} 
				g := \theta(x_v-x_w) &= \theta(x_v) -\theta(x_w) \\ 
				&= \theta(x_v) - \theta(x_{v_1^s}) + \theta(x_{v_1^s}) - \dots -\theta(x_{v_{n_s}^s}) + \theta(x_{v_{n_s}^s}) - \theta(x_w).
			\end{align*}  
			Since $\theta\in D(\A)$, the defining condition can be applied to the edges in the path.
			
			We can thus find polynomials in $S$, such that for a suitable choice of polynomials $p_{u,t}, q_{u,t} \in \mathbb{K}[x_1,\dots, x_\ell]$, we can express $g$ as
			
			\begin{equation} \label{equ:g_ideals} 
				\begin{split} 		
					g &= (x_v-x_{v_1^s})\cdot p_{v, v_1^s} +(x_{v_1^s}-x_{v_2^s})\cdot p_{v_1^s, v_2^s} + \dots + (x_{v_{n_s}^s}-x_w)\cdot p_{v_{n_s}^s, w} \\ &= (x_v-x_{v_1^s})\cdot p_{v, v_1^s} +(x_{v_1^s}-x_v)\cdot p_{v_1^s, v_2^s}+(x_v-x_{v_2^s})\cdot p_{v_1^s, v_2^s} + \dots + (x_{v_{n_s}^s}-x_v)\cdot p_{v_{n_s}^s, w} + (x_v-x_w)\cdot p_{v_{n_s}^s, w} \\ &=  \sum_{i = 1}^{n_s} (x_v-x_{v_i^s}) \cdot  (p_{v_{i-1}^s,v_i^s}-p_{v_i^s,v_{i+1}^s}) + (x_v-x_w)\cdot p_{v_{n_s}^s, w}\\ &=  \sum_{i = 1}^{n_s+1} (x_{v_i^s}-x_v)\cdot q_{v, v_i^s}.
				\end{split} 
			\end{equation} 
			where $v_0^s = v, v_{n_s+1}^s = w ~\text{for all}~ 1\leq s \leq k$.

			Set $\overline{S} := S/(x_v-x_w)$ and define variables $z_v = x_v, z_u = x_u-x_v$ for all $u \not= v$.

			Consider the map $\phi: \overline{S} \rightarrow K[z_1,\dots, z_{\ell}]/(z_w)$ defined by  $x_u \mapsto z_u$. 
			It is a ring isomorphism with inverse \[\varphi(z_u) = \begin{cases} x_u+x_v & u \not= v \\ x_v & u = v. \end{cases} \]
			
			Let $I_s = ((x_v-x_{v_1^s}),\dots,(x_v-x_{v_{n_s}^s}))/(x_v-x_w)$, so by equations (\ref{equ:g_ideals}) \[\overline{g} \in \bigcap_{s = 1}^{k} I_s, ~\text{in particular }\phi(\overline{g}) \in \bigcap_{s = 1}^{k} (z_{v_{1}^s},\dots, z_{v_{n_s}^s}).\]
			
			We now apply Lemma \ref{lem: mongen} for the special case where the variable sets are disjoint, therefore  $\langle Z_{\text{trans}} \rangle = \prod \phi(I_s)$ and thus $\phi(\overline{g}) \in \prod_{s= 1}^{k} \phi(I_s)$. 
			
			Since $\phi$ is an isomorphism, we get that $\overline{g} \in \prod_{s = 1}^{k} I_s$ and it follows that either $g = 0$ or $\deg(g) \geq k$. So, $\overline{g} = 0$ and $\theta(x_v-x_w) \in (x_v-x_w)\cdot S$. \qedhere
			
		\end{proof}
		
		From this proposition, we recover the basis of the derivation module of the braid arrangement, since the braid arrangement is the graphic arrangement  corresponding to the complete graph $K_\ell$, which has connectivity $\ell-1$.

		This result hinging on connectivity tells us that it might be useful to consider minimal separators for the generation of a derivation module.

		\medskip
		For a separator $T$ and a connected component $C \in \mathscr{C}(T)$, recall the derivations  \[\theta_{C}^T = \sum_{j \in C} \prod_{t \in T} (x_j-x_t) \cdot  D_j \] of Definition \ref{def:sep_based_der}. 
		
		\begin{lemma} 
			Let $G$ be a simple graph. The separator based derivation $\theta_{C}^T$ is in the derivation module $D(\A(G))$ for any separator $T$ of $G$ and $C \in \mathscr{C}(T)$.  
		\end{lemma} 
		\begin{proof} We need to show $\theta_C^T(x_i-x_j) \in (x_i-x_j)\cdot S$ for all edges $\{i,j\} \in G$.  For $i,j$ both not in $C$ this value is trivial and there are only two cases left to consider: 
			\begin{enumerate} 
				\item $j \in T, i \in C$, then $\theta_{C}^T(x_i-x_j) = \prod_{t \in T} (x_i-x_t)$ which is in $(x_i-x_j)\cdot S$.
				\item $i,j \in C$, then \[\theta_{C}^T(x_i-x_j) = \prod_{t \in T} (x_i-x_t) - \prod_{t \in T} (x_j-x_t)\] which is in $(x_i-x_j)\cdot S$. \qedhere
			\end{enumerate}  
		\end{proof} 
		\medskip  
		
		For the principle which we want to apply in constructing generating sets for general graphs, consider the following example: 
		
		\begin{example} \label{ex:braid_arrangement} 
			
			Recall that $\theta_i = \sum_{j =1}^\ell x_j^i \cdot D_j$. In the braid arrangement case, the coefficients for $\theta_i$, $i < \ell$ to get the derivations $\theta_k$, $k \geq \ell$ are the symmetric polynomials. We define \[\varphi_i := \sum_{k = 1}^\ell \prod_{j < i} (x_i-x_j) \cdot D_j = \sum_{k = 0}^i (-1)^{k} f_{i-k}^{\text{sym}}\cdot \theta_k \] where 
			$f_{i-k}^{\text{sym}}$ is the elementary symmetric polynomial of degree $i-k$ in the variables $x_1,\dots,x_{i-1}$. The derivations $\varphi_i$ are thus generated by the $\theta_i$ and vice versa. We use the identity: 
			\[x^k-y^k = (x-y)\cdot \sum_{i = 1}^k y^{i-1} x^{k-i}.\]
			
			Now for $k \geq \ell$, write 
			\begin{align*} 
				\theta_k &= x_1^k \cdot \theta_{0} + \sum_{i = 1}^\ell (x_i^k-x_1^k) D_i \\ 
				& = x_1^k \cdot \varphi_{0} + \sum_{j = 1}^k x_2^{k-j} x_1^{j-1} \cdot \varphi_{1}  + \sum_{i = 1}^\ell (x_1-x_i) \sum_{j = 1}^k x_1^{j-1}\cdot  (x_i^{k-j} - x_2^{k-j}) \cdot D_i \\ 
			\end{align*} 
			Iteratively, we get that $\theta_k$ is generated by the $\varphi_i$ and thus also by $\theta_i, i < \ell$. 
		\end{example} 
		
		The main goal of the following statements is to pave the way to finding a generating set with infinitely many derivations based on separators. We then prove that it is generating $D(\A(G))$ by the same method we used to prove Proposition \ref{connectedprop} and that we can find a finite subset which generates this specific infinite set. We start by introducing the relevant derivations. Recall that $S = \mathbb{K}[x_1,\dots,x_\ell]$. 
		
		\begin{definition} 
			Let $S[y]$ be the polynomial ring in one variable with polynomials in the $x_i$ as coefficients. For $p \in S[y]$ define the derivation 
			\[\theta_C^{T, p} = \sum_{k \in C} \prod_{t\in T} (x_k-x_t) \cdot p(x_k)\cdot D_k.\] 
		\end{definition}
		The complication in graphs that are not trees or complete is that derivations of this form need to be generated, since they lie in the derivation module $D(\A)$. We first start with the observation that these derivations combined with the $\theta_i$ generate the module:  
		
		\begin{proposition} \label{genprop} 
			Let $\mathcal{A}$ be a graphic arrangement. Then, all derivations $\theta$ contained in the derivation module $D(\mathcal{A}(G))$ are an $S$-linear combinations of the following: 
			\begin{enumerate} 
				\item $\theta_i, i \in \mathbb{Z}_{\geq 0}$ 
				\item $\theta_{C}^{T,p}$, for all  $T \in \mathscr{S}(G), C \in \mathscr{C}(T), p \in S[y]$. 
			\end{enumerate} 
		\end{proposition}

		\begin{proof} 
			Denote the set of the listed elements as $ \mathscr{M}_{\A(G)}$.  
			
			We prove the claim by induction over the number of non-edges in the graph: We know that $D(\A(K_\ell))$ (zero non-edges) is generated by the list, in that case, there are no minimal separators. So assume the claim has been shown for $\A(G)$ and consider $\A(G)'$ with distinguished hyperplane $H_0$ corresponding to the edge $e_0 = \{v,w\}$. Define $G' := G \backslash \{e_0\}$ as the graph with $e_0$ removed, such that $\A(G)' = \A(G')$. 
			
			\medskip 
			All elements of $ \mathscr{M}_{\A(G')}$ clearly lie in $D(\mathcal{A}(G'))$. 
			Let $\theta \in D(\A(G'))$ and let $m$ be the number of (not necessarily disjoint) paths $P_1,\dots,P_m$ between $v$ and $w$ in $G'$. For each $1\leq i \leq m$, we get a representation for  $g := \theta(x_v-x_w)$ by means of 
			
			\begin{align*} 
				\theta(x_v-x_w) &= (x_v-x_{v_1^i})\cdot p_{v, v_1^i} + \dots + (x_{v_{n_i}^i}-x_w)\cdot p_{v_{n_i}^i, w} \mod (x_v-x_w)\\ &= \sum_{j = 1}^{n_i} (x_v-x_{v_j^i})\cdot q_{v, v_j^i}
			\end{align*} 
			
			In the notation of Proposition \ref{connectedprop}, we define $z_u := (x_v-x_u)$ and again consider the map
			$\phi$. It follows in a similar fashion that  \[\overline{g} \in \bigcap_{s = 1}^{m} I_s,\] in particular $\phi(g) = \bigcap_{s = 1}^{m} (z_{v_1^s},\dots, z_{v_{n_s}^s})$. 
			
			In this case, we do not have disjointness of the sets anymore. For a set of vertices $U$ define the monomial \[\mathbf{z}_U := \prod_{u \in U} z_u\]and the set of subsets $M = \{U\subseteq V(G')~\vert~\text{for all } i: U \cap P_i \not= \emptyset\}$.
			
			By Lemma \ref{lem: mongen}, $\overline{g}$ is generated by the monomials $\mathbf{z}_U, U\in M$ and thus it can be written as $g = \sum_{U \in M} \mathbf{z}_U \cdot q_U $ for some polynomials $q_U \in S[y]$. 
			
			Now by definition of the sets $U$, each of them contains a minimal separator $T_U \subseteq U$, namely a $(v,w)$-separator. Denote by $C_U$ the connected component of $G'\backslash T_U$ containing $v$ and define \[p_U = \prod_{u \in U\backslash T_U} z_u \cdot q_U \in S[y].\]
			
			Thus, we can consider the derivation $\theta' = \theta - \sum_{U \in M} \theta_{C_U}^{T_U, p_U}$. Per definition, we have that $\overline{\theta(\alpha_{H_0})} = \overline{g} = \overline{\sum_{U \in M} \theta_{C_U}^{T_U, p_U}(\alpha_{H_0})}$ in $\overline{S}$. Therefore, it holds that $\theta'(\alpha_{H_0}) \in \alpha_{H_0}\cdot S$ and as a linear combination of two derivations in $D(\A'(G))$, $\theta$ is itself in that module. 
			
			Hence, $\theta'\in D(\A(G))$, which by induction is generated by $ \mathscr{M}_{\A(G)}$. By closer examination, we see that the elements of $\mathscr{M}_{\A(G)}$ are either also in $\mathscr{M}_{\A(G')}$ or generated by that set, since the only thing that can happen is that one of the vertices is part of a separator and by removing the edge between $v$ and $w$, the separator is not minimal any more.
			
			Analogously for the connectivity: $G'$ might have connectivity one less that $G$, but this means that there is a minimal separator $T$ with $\vert T \vert =  \kappa(G)$, such that $\theta_{\kappa(G)}$ is generated by $\theta_0,\dots,\theta_{\kappa(G)-1}$ and $\sum_{C\in \mathscr{C}(T)} \theta_C^T$. 
			
			So, $\mathscr{M}_{\A(G)} \subset \langle \mathscr{M}_{\A(G)'}\rangle$ 
			and thus $\theta \in \langle \mathscr{M}_{\A'(G)}\rangle = D(\A'(G))$. \qedhere
			
		\end{proof} 
		
		We can rephrase this proposition in an elegant way by introducing a new element, the \emph{empty vertex}, that is not connected to any other vertex and not labelled. 
		
		In this way $\theta_0 = \theta_{G}^\emptyset, \theta_k = \theta_{G}^{\emptyset, x^k} $ and the proposition can be rephrased using only the second type of derivation.
		
		 We illustrate Proposition \ref{genprop} with a very natural example class of derivations.
		
		\begin{example}\label{ex:nui} 
			Define the \emph{neighbourhood derivations} of $D(\mathcal{A}(G))$ as \[\sigma_i = \prod_{j \in N(i)} (x_i-x_j) \cdot D_i\] for $i \in V(G)$. If $\{i,j\}\in E(G)$ for all $j \not= i$, then $\sigma_i$ is generated by the $\theta_k, k \in \mathbb{N}$. If that is not the case, $\sigma_i$ is generated by the $\theta_C^{T,p}$: Let $T\subseteq N(i)$ be a minimal separator (it exists, since $N(i) $ is separating $i$ from at least one other vertex), then $\sigma_i = \theta_{i}^{T, q_i}$ with \[q_i(y) = \prod_{j \in N(i)\backslash T} (y-x_j).\]
			
		\end{example}

		The sets of Proposition \ref{genprop} are not finite, so the next step is to see how we can generate these through a finite set. We start with an example to illustrate the task.

		\begin{example} \label{ex:desc_chain} 
			
			In order to generate the derivations of the form $\theta_C^{T,p}$ for general graphs we cannot rely on generators arising from minimal separators only, to see this, consider the graph $G$ on 4 vertices in Figure \ref{fig:ex1}. 
			
			This graph is chordal and has only one minimal separator, namely the set containing only the  vertex 4, yielding two connected components. Thus $\kappa(G) = 1$ and we can construct three derivations of the form previously established ($\theta_0, \theta_{\{1\}}^{\{4\}}, \theta_{\{2,3\}}^{\{4\}}$). Yet any basis needs to contain exactly four elements.  
			
			One possible choice of basis is the following set of derivations, written as the transpose of the coefficient matrix  (see Def. \ref{def:matcoef}):

			\[	\begin{bmatrix} 
				1 & 1 & 1 & 1 \\
				(x_1-x_4) & 0 & 0 & 0 \\ 
				0 & (x_2-x_4) & (x_3-x_4) & 0 \\
				0 & (x_2-x_4)\cdot (x_2-x_3)& 0 & 0 
			\end{bmatrix} 
			\] 
			We can check that this is a basis using Saito's criterion. Note that the first three derivations are the ones described above. The choice of $(x_2-x_4)\cdot (x_2-x_3)\cdot D_2$ as fourth generator may seem arbitrary in the sense that we could also have chosen $(x_3-x_4)\cdot (x_3-x_2)\cdot D_3$. 
		\end{example} 
		We want to formalize this type of derivation with the following definition. 
		
		\begin{figure}[H]
			\includegraphics[width=0.2\textwidth]{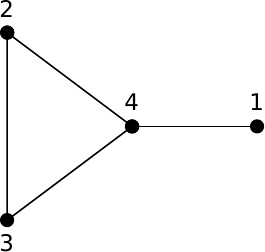} 
			\caption{Example of a chordal graph with 4-dimensional $\A(G)$-derivation module} 
			\label{fig:ex1}
		\end{figure}

		\begin{definition}\label{separatorchain} 
			Let $T$ be a minimal separator of a graph $G$ and $C\in \mathscr{C}(T)$ with $\vert C \vert = k$. 
			
			Let $v_1,\dots,v_k$ be an ordering on the vertices of $C$ and define the sets $C_i := \{v_1,\dots,v_i\}$ and set $T_i = \bigcup_{j = 1}^i N(v_j)\backslash C_i$. 
			Define the derivations \[\theta_{C_i}^{T_i} = \sum_{j \in C_i} \prod_{t \in T_i} (x_j-x_t) \cdot D_j.\]
			Call $(T_i, C_i)_{1\leq i \leq k}$ a descending chain of $(T,C)$. 
		\end{definition}

		\begin{remark} The definition of the $T_i$ and $C_i$ is dependent of the chosen ordering, thus in general, there is no unique descending chain for a pair $(T,C)$. In Example \ref{ex:desc_chain}, the two derivation pairs $(x_2-x_4)\cdot (x_2-x_3)\cdot D_2$ and $(x_3-x_4)\cdot (x_3-x_2)\cdot D_3$ form two different descending chains with the pair $(T,C) = (\{4\}, \{2,3\})$. 
		\end{remark} 
		
		\begin{lemma} \label{polynomial_generation} 
			Let $T$ be a minimal separator of a graph $G$ and $C\in \mathscr{C}(T)$ with $\vert C \vert = k$. 
			
			Then, with the previous definition, the set $\{\theta_{C_i}^{T_i}, 1 \leq i \leq k\}$ generates all derivations in $D(\A(G))$ of the form $\theta_{C}^{T, p(x)}(x_m) = \prod_{t\in T} (x_m-x_t) \cdot p(x_m)$ for all $p \in S[y]$. 
		\end{lemma} 
		\begin{proof} 
			Note that $T_{i-1} \subseteq T_i\cup \{v_i\}$.
			We only need to show that the set generates all elements $\theta_C^{T, y^m}$ for all $m \in \mathbb{N}$. We prove this by induction over the sequence of subsets, namely, we prove that for every $p_i \in S[y]$, there exist $p_0,\dots, p_{i+1} \in S[y]$, such that $\theta_{C_i}^{T_i, p_i} \in \langle \theta_{C_j}^{T_j, p_j} 0 \leq j \leq i-1, \theta_{C_i}^{T_i} \rangle$. 
			
			First, $C_1$ only contains one vertex, thus the derivation only has one non-zero entry and we can multiply it with any polynomial. 
			Now consider $\theta_{C_i}^{T_i}$ and the vertex $v_i$ with variable $x_i$. 
			Compute 
			\[x_i^m \cdot \theta_{C_i}^{T_i} = \theta_{C_i}^{T_i, y^m} + \sum_{j \in C_i} \prod_{t \in T_i} (x_j-x_t)\cdot (x_j^m-x_i^m)D_j\] 
			The second part can be written as 
			\begin{align*}
				\sum_{j \in C_i} \prod_{t \in T_i} (x_j-x_t)\cdot (x_j^m-x_i^m) &=  \sum_{j \in C_i} \prod_{t \in T_i} (x_j-x_t)\cdot (x_j-x_i)\cdot q(x_j)  \\ &= \sum_{j \in C_{i-1}} \prod_{t \in T_i} (x_j-x_t)\cdot (x_j-x_i)\cdot q(x_j) \\ &= \sum_{j \in C_{i-1}} \prod_{t \in T_{i-1}} (x_j-x_t)\cdot q'(x_j)\\ 
			\end{align*} 
			
			where $q,q'$ are some polynomials in $S[y]$ and the last equality comes from the fact that $T_{i-1}$ contains $j$ and potentially less elements than $T_i$, which can be wrapped up in $q'$. 
			
			This is a derivation that is, by induction, generated by $\theta_{C_j}^{T_j}$ for $j \leq i-1$ and we are finished. 
		\end{proof}
		\bigskip

		\medskip
		We now summarize which elements need to be in a set of generators according to our previous remarks and check the remaining condition according to Proposition \ref{genprop} and Lemma \ref{polynomial_generation}. 
		\begin{corollary} \label{cor:thetai}
			Let $G$ be a $(k+1)$-connected graph. Any set $\Theta_G$ containing $\theta_0,\dots,\theta_{k}$ and a descending chain for every minimal separator $T$ of $G$ and every $C\in \mathscr{C}(T)$ generates all derivations of the form $\theta_i, i > k$. 
		\end{corollary}

		Note that using the notion of an empty vertex, we get that $\theta_1,\dots,\theta_k$ is just the beginning of a chain for the separator $\emptyset$ and the chain can be amended by any chain of a smallest minimal separator in $G$, such that this would not need to be specified in the definition. 
		
		\begin{proof} 
			The case of $G$ being the complete graph was explored in Example \ref{ex:braid_arrangement}. 
			Let $G$ be non-complete. Since $\kappa(G) = k+1$, this means that there exists a minimal separator $T$ with $k+1$ vertices. 
			The sum of all derivations $\theta_C^T, C\in \mathscr{C}(T)$ is \[\theta_\Sigma = \sum_{i \in G\backslash T} \prod_{t\in T} (x_i-x_t) D_i\]
			which is just the derivation $\theta_{G\setminus T}^{T, p}$ with the polynomial \[ p(y) = \prod_{t\in T} (y-x_t) \in S[y] \] in every entry. 
			
			We can rewrite this homogeneous derivation as $\theta_\Sigma = \theta_\Sigma^{(0)} + \dots + \theta_\Sigma^{(k+1)}$, where $\theta_\Sigma^{(i)} $ with $i < k+1$ are generated using $\theta_i, 0 \leq i \leq k$. Subtracting them, we are left with the derivation $\theta_\Sigma^{(k+1)}$ which has $x_i^{k+1}$ in every entry, we thus have generated $\theta_{k+1}$. 
			
			From Lemma \ref{polynomial_generation}, we know that all derivations of the form $\sum_{C \in \mathscr{C}(T)} \theta_C^{T, x^p}$ are generated by $\Theta_G$ and with the same argument as above, we get that all $\theta_p$ are generated by this set. \qedhere
			
		\end{proof} 
		
		Finally, we have the following: 
		\begin{theorem}\label{generating_corollary} 
			Let $G$ be a graph and $\mathcal{A}(G)$ the associated graphic hyperplane arrangement. Then any set $\Theta_G$ fulfilling the conditions of Corollary \ref{cor:thetai} is a generating set of $D(\mathcal{A}(G))$. 
		\end{theorem}
		
		\begin{example} \label{ex:trees} 
			For a tree $\mathscr{T}$ it holds that $\mathscr{S}(\mathscr{T}) = \{v \in V(\mathscr{T})\vert \deg(v) \geq 2\}$.  
			
			Choose a vertex $v_0 \in \mathscr{T}$ with $\deg(v_0) \geq 2$, define $T_0 = \{v_0\}$ and denote the connected component of a minimal separator $T \not= T_0$ of $\mathscr{T}$ that contains $v_0$ as $C_0^T$. We can write the basis of $D(\A(\mathscr{T}))$: 
			
			\[\Theta_{\mathscr{T}} = \{\theta_0\}\cup \big\{\theta_C^{T_0}~\vert ~C \in \mathscr{C}(T_0)\big\}\bigcup_{T \in \mathscr{S}(\mathscr{T})\backslash{T_0}} \big\{\theta_C^{T} ~\vert~C\in \mathscr{C}(T)\backslash C_0^T\big\}  \]

			We want to use Saito's criterion (see Theorem \ref{thm:saito}).  
			First, we observe that $\vert \Theta_{\mathscr{T}}\vert = \vert V(\mathscr{T}) \vert$, since we can consider $\mathscr{T}$ to be rooted with root $v_0$ and the set $\big\{\theta_C^{T} ~\vert~C\in \mathscr{C}(T)\backslash C_0^T\big\}$ has size equal to the number of descendants of $v_0$. 
			
			One verifies that we thus get $n-1$ many derivations and together with $\theta_0$ (in this way corresponding to $v_0$), they form a $n\times n$-matrix.

			To compute the determinant, we use the Leibnitz formula and we note the following: 
			\begin{enumerate} 
				\item for each leaf  $z$ of the tree, there is a derivation of the form 
				$(x_w-x_z)\cdot D_z$, so the only non-zero entry in the respective (let's say $y$-th) row of the matrix would be in the $z$-th column. All permutation with $\sigma(z) \not= y$ vanish. 
				\item A derivation corresponding to a connected component (which in this case is a subtree) for a separator  has non-zero entries for all vertices in the subtree, but iteratively, we get that a permutation already has fixed values for all entries, except the root of the subtree. 
				This means that there is again only one non-zero choice to make and it is the edge connecting the subtree with its root one level up. 
				\item Finally, $\sigma(v_0)$ is the only value left to choose and the only non-zero choice is the row corresponding to $\theta_0$. 
			\end{enumerate} 
			Combining these observations, there is only one specific permutation not vanishing and the product of the matrix entries is (up to sign) the product of all $(x_i-x_j)$ for $\{i,j\}$ in the tree, thus a defining polynomial $Q$ of $\mathcal{A}(\mathscr{T})$. 
			
		\end{example} 
		\begin{example} 
			The $\ell$-antihole graph is defined as the complement graph of the cycle graph $C_\ell$. In \cite{abe2023projective}, we proved that the set \[\theta_0, \dots ,\theta_{\ell-3}, \sigma_i, 1 \leq i \leq \ell\]
			is a generating set, $\sigma_i$ defined as in Example \ref{ex:nui}. The connectivity of the $\ell$-antihole is $\ell-3$ and the minimal separators are the neighbourhoods of the vertices, thus this statement immediatly follows from Theorem \ref{generating_corollary}. 
		\end{example} 
		
		\subsection*{The separator poset} 
		In this section we identify the generating sets of $D(\A)$ with posets fulfilling certain conditions.
		
		\begin{definition}
			Let $G$ be a graph. The \emph{separator poset} $\mathscr{P}(G)$ is the set \[\{(T,C)~\vert~ T \in \mathscr{S}(G), C \in \mathscr{C}(T)\}\] ordered by inclusion of the connected components ($(T,C) \leq (T', C') ~\text{if}~C \subseteq C'$). 
		\end{definition}
		We can add a (possibly empty) set of elements $(T,C)$ to $\mathscr{P}(G)$, where $T$ is a non-minimal separator and $C$ is a union of components of $\mathscr{C}(T)$. Call this poset an \emph{augmented separator poset} of $G$. We identify an element of the form $(T,C)$ with the derivation $\theta_C^T$. 
		
		\begin{definition} 
			Let $G$ be a graph with separator poset $\mathscr{P}(G)$ and $\mathcal{Q}$ an augmented separator poset of $G$. We add elements $(T', C')$ generated (as derivations) by elements of $\mathcal{Q}$, such that there exists an element $p \in \mathscr{P}$ with $p > (T',C')$. Call this new poset a \emph{generated $\mathcal{Q}$-poset}. 
		\end{definition} 
		
		Let $(T,C) \in \mathscr{P}(G)$ with $\vert C \vert = k$. Call a chain $(T_1, C_1) < \dots < (T_k,C_k) = (T,C)$ a \emph{complete chain} in $\mathscr{P}(G)$. 
		With this definition, we can define the notion of a generating set of $D(\A)$ in terms of posets: 
		\begin{lemma} 
			Let $G$ be a graph and let $\mathcal{Q}$ be an augmented separator poset. The elements in $\mathcal{Q}$ form a generating set of $D(\A)$ together with the set $\{\theta_i, 0 \leq i \leq \kappa(G)\}$ if there exists a generated $\mathcal{Q}$-poset, such that every element of $\mathscr{P}(G)$ is the maximal element of a complete chain in this poset. In this case, call $\mathcal{Q}$ \emph{complete}.  
		\end{lemma} 
		\begin{proof} 
			This follows immediately from Theorem \ref{generating_corollary}, since the complete chains of the poset correspond to the descending chains of the theorem. 
		\end{proof} 
		\begin{example} \label{ex:7_without_chord} 
			Consider the graph $G$ in Figure \ref{fig:7_without_chord}. For readability, the tuples $(T,C)$ in the illustrations are written as $[t\in T], \{v \in C\}$. The separator poset of $G$ does not contain a complete chain for each tuple, but there is a generated $\mathscr{P}(G)$-poset, which has a complete chain for each element of $\mathscr{P}(G)$, therefore, it is complete. 
			\begin{figure}
				\includegraphics[width=0.8\textwidth]{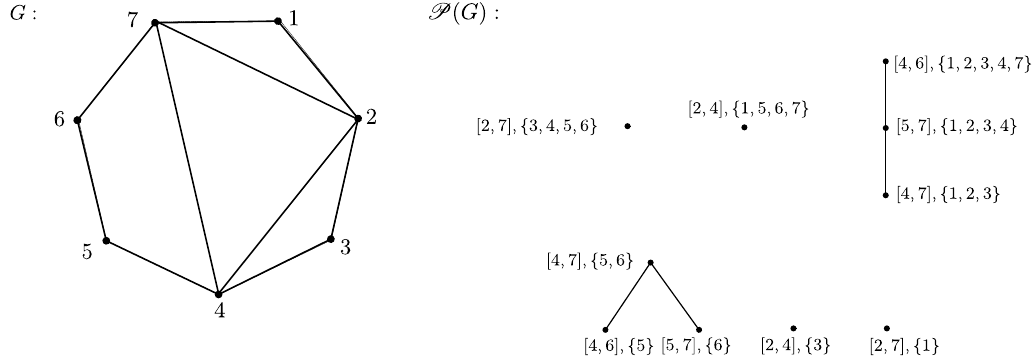} 
				\caption{The graph $G$ with its unique separator poset.}
				\label{fig:7_without_chord}
			\end{figure}
		\end{example} 
		We note here the characterization of minimality in terms of the poset structure: 
		
		\begin{definition}
			A complete poset $\mathcal{Q}$ corresponds to a minimal generating set if no proper subposet of $\mathcal{Q}$ is complete. 
		\end{definition}
		\section{On minimal generating sets} \label{minimality} 
		
		The goal of this section is to describe minimal generating sets in terms of the poset structure established in the previous section. Since the generators are homogeneous, the derivation degree sequence (see Definition \ref{def:der_deg_seq}) is unique and thus, we are able to derive some general statements about it using our construction.

		\medskip 
		\noindent 
		In order to find redundant elements in a set of generators, we look at how we can generate derivations through other separator-based derivations. These generating methods can be described in terms of the poset structure as we explain in the lemmas below: 
		
		\begin{lemma} \label{lem:poset_gen} 
			Consider a derivation $\theta^{T}_{C}$ for $T$ a separator of $G$, $C\in \mathscr{C}(T)$. If there exist tuples $(T_i, C_i), 1\leq i \leq k$, with $T_i \in \mathscr{S}(G), C_i \in \mathscr{C}(T_i)$ and $A|B$ a partition of $[k]$ such that \[\bigcup_{i \in A} C_i \setminus \bigcup_{i \in B} C_i = C~\text{and}~\bigcup_{i = 1}^k T_i \subseteq T,\] then $\theta^{T}_{C}$ is generated by the derivations corresponding to $\bigcup_{i = 1}^k \{(T_i, C_i)\} \cup \mathscr{D}(T_i, C_i)$, where $\mathscr{D}(T_i, C_i)$ is a descending chain for $(T_i, C_i)$. 
		\end{lemma} 
		\begin{proof} 
			Let $D = \bigcup_{i \in A} C_i$ and $E = \bigcup_{i \in B} C_i$, i.e. $C = D \setminus E$. 
			
			First, assume there exist two minimal separators $T_{i_1}, T_{i_2}$ with $i_1,i_2$ either both in $A$ or  both in $B$, such that $C_{i_1} \cap C_{i_2} \setminus T$ is non-empty. This in turn would mean that either $T_{i_1}\cap C_{i_2}$ or $T_{i_2}\cap C_{i_1}$ is not empty, but then the vertices of such an intersection would be in $T$. 
			
			The $C_{i_j}\setminus T$ in $D$ and $E$ respectively are disjoint and we can write 
			\[\theta_{C}^T = \sum_{i \in A} \theta_{C_i}^{T_i, \prod_{t \in T\backslash T_i} (y-x_t)} - \sum_{i \in B} \theta_{C_i}^{T_i, \prod_{t \in T\backslash T_i} (y-x_t)},\] where the derivations of the sum are generated by the minimal separators and their descending chains: Since the components are disjoint, each entry $i$ in the derivation $\theta_C^T$ appears only once  in $D$ or $E$ respectively, and thus is only non-zero in one of the derivations. Multiplying this non-zero expression with the polynomial  $\prod_{t \in T\backslash T_i} (y-x_t)$ yields the desired entry. 
		\end{proof} 
		\begin{example} 
			We can apply Lemma \ref{lem:poset_gen} to the graph of Example \ref{ex:7_without_chord}: Consider the derivation $\theta_{\{1,5,6\}}^{\{2,4,7\}}$, it is generated by the derivations $ \theta_{\{5,6\}}^{\{4,7\}}, \theta_{\{6\}}^{\{5,7\}},\theta_{\{1\}}^{\{2,7\}}$ by \[\theta_{\{1,5,6\}}^{\{2,4,7\}} = (x_1-x_4) \cdot \theta_{\{1\}}^{\{2,7\}} + \theta_{\{5,6\}}^{\{4,7\}, (y-x_2)}\] and \[\theta_{\{5,6\}}^{\{4,7\}, (y-x_2)} = (x_5-x_2)\cdot \theta_{\{5,6\}}^{\{4,7\}}+(x_6-x_4)\cdot \theta_{\{6\}}^{\{5,7\}}.\] 
			
			See Figure \ref{fig:7_without_chord_gen} for a complete generated $\mathscr{P}(G)$-poset.  
			
			\begin{figure}
				\includegraphics[width=0.8\textwidth]{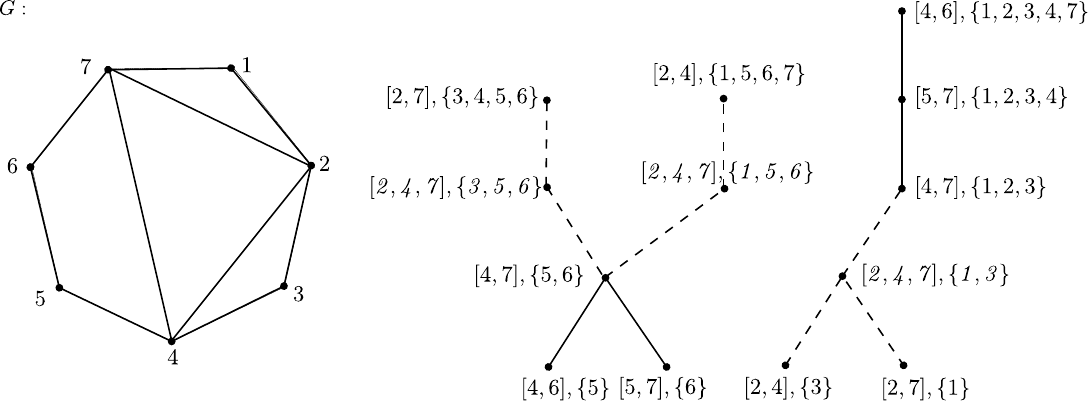} 
				\caption{The graph $G$ with a generated $\mathscr{P}(G)$-poset. The elements that are generated by $\mathscr{P}(G)$ are in italic. Note that  $\mathscr{P}(G)$ is complete.} 
				\label{fig:7_without_chord_gen}
			\end{figure}
		\end{example}

		The statement of this lemma can be summarized as follows: If we have an augmented poset $\mathcal{Q}$ of a graph $G$ and an element $p$ of rank $m$, there might exist a set of elements in $\mathcal{Q}$ whose ranks are a partition of $m$, where the corresponding derivations generate the derivation corresponding to $p$. Thus, the latter derivation would be redundant in a generating set containing them all. 
		
		We next look at a special case of of generation through the lemma. 
		\begin{example} \label{lem:gen_compl}
				Let $(T,C)$ be a tuple with $T$ a separator and $C\in \mathscr{C}(T)$. Define the complement of $C$ with respect to $T$ as $G\backslash (T\sqcup C)$ and denote it by $\overline{C_T}$.  
				
			The derivations $\theta_C^T, \theta_i, i \in \mathbb{N}$ generate $\theta_{\overline{C_T}}^T$ as we can write \[\theta_{\overline{C_T}}^T = \sum_{j\in V(G)} \prod_{t \in T} (x_j-x_t) D_j - \theta_C^T\] where the first derivation is generated by the $\theta_i$. 
			Thus, following Corollary \ref{cor:thetai}, if we have $\theta_0,\dots, \theta_{\kappa(G)}$ in a generating set and a minimal separator of minimal cardinality $T_0$, for which all descending chains $(T_0, C_0), C_0 \in \mathscr{C}(T)$ exist, then for each $(T,C)$ we can automatically generate the complement $(T, \overline{C_T})$. 
		\end{example} 
		
		The next lemma makes clear how minimality of a separator factors into minimality considerations of a generating set.

		\begin{lemma} \label{lem:non_min_gen} 
			Let $T \in \mathscr{S}(G), C\in \mathscr{C}(T)$, then $T$ is never generated through means of other minimal separators as described in Lemma \ref{lem:poset_gen}. 
		\end{lemma} 
		\begin{proof} 
			Assume that there exist tuples $(T_i, C_i), 1\leq i \leq k$ with $T_i \in \mathscr{S}(G), C_i \in \mathscr{C}(T_i)$ which meet the conditions of Lemma \ref{lem:poset_gen} and $T_i \not= T$ for all $i$. It follows that $T_i \subsetneq T$ for all $i$. Since $T$ is minimal, there must be $a \in C, b \notin C$, such that $T$ separates $a,b$ and no proper subset of $T$ separates $a,b$. This in turn means that since $a \in C_i$, $b \in C_i$ or $b \in T_i$ for all $i$ and thus $b \in \bigcup_{i = 1}^k C_i$ or $b \in \bigcup_{i = 1}^k T_i$, a contradiction to the assumption. 
		\end{proof} 
		
		This lemma tells us that if we want to start with a set of derivations to obtain a generating set, the derivations corresponding to the elements of $\mathscr{P}(G)$ are a good starting point, because these elements will not be generated by derivations associated to other separators. 
		
		\begin{lemma} 
			In a generating set $\Theta$ corresponding to a complete, augmented $\mathscr{P}(G)$-poset, a generator may only be generated through other elements of $\Theta$ by means of Lemma \ref{lem:poset_gen}. 
		\end{lemma} 
		\begin{proof} 
			Assume that there is a derivation $\theta_C^T$ generated by some $\theta_{C_1}^{T_1}, \dots ,\theta_{C_k}^{T_k}$ and the $\theta_i$. By construction of the separator-based derivations and their homogeneity, it follows that $T_i \subseteq T$ for all $T_i$. For all $\theta_{C_i}^{T_i}$, define $\theta_{C_i}' = \theta_{C_i}^{T_i, \prod_{t \in T\setminus T_i} (y-x_t)}$. Because the derivations are in $D(\A)$, if $\theta_{C_i}'$ has a nonzero entry for a vertex of some component $D \in \mathscr{C}(T)$, it has a non-zero entry for all vertices of $D$. Therefore, with the same argument as in the proof of Lemma \ref{lem:poset_gen}, we get that the derivations $\theta_{C_i}'$ that share entries in one component $D \in \mathscr{C}(T)$, they either agree on all entries in $D$ or their sets of indexes of non-zero entries are disjoint. Thus, they meet the condition of Lemma \ref{lem:poset_gen}. 
		\end{proof} 

		A good heuristic to find a minimal generating set, would be as follows: Let $G$ be a graph. 
		\begin{enumerate} 
			\item Construct $\mathscr{P}(G)$ and define $\mathcal{Q}(G) := \mathscr{P}(G)$. 
			\item As long as there is no generated $\mathcal{Q}(G)$-poset such that a minimal separator $T_{\min}$ of minimal cardinality has descending chains for all elements $(T_{\min}, C), C \in \mathscr{C}(T_{\min})$, augment the poset by one element, adding to one of the incomplete chains. 
			\item For each minimal separator $T \not= T_{\min}$, remove one element $(T, C) C \in \mathscr{C}(T)$ that is not in a descending chain of $T_{\min}$ (see Lemma \ref{lem:no_chains}). 
		\end{enumerate} 
		\begin{lemma} \label{lem:no_chains} 
			For $T, T_{\min}$ as described above, there always exists an element $(T,C)$ which is not part of a complete chain of any $(T_{\min}, C_{\min})$ in the complete poset. 
		\end{lemma} 
		\begin{proof} 
			If $T_{\min} \not\subset T$, then there exists an element in $\mathscr{C}(T)$ that contains vertices of $T_{\min}$ and can thus not be part of any complete chain of $T_{\min}$. If $T_{\min} \subset T$, then $\mathscr{C}(T_{\min})\cap \mathscr{C}(T) \not= \emptyset$, because $T_{\min}, T$ are both minimal. We can chose $C$ as one of the elements of this intersection. 
		\end{proof} 
		\subsection*{Examples of posets} 
		We consider here the example of $k$-antiholes and an example of a cycle graph with two chords that intersect, as opposed to the chords in Example \ref{ex:7_without_chord}.
		\begin{example} 
			The anti-hole graphs $\overline{C_k}, k \geq 5$ are a graph class which was of special interest in \cite{abe2023projective}. In this case, the elements of $\mathscr{P}(G)$ already form descending chains for each element, so no generators need to be added to it at all. 
			See Figure \ref{fig:antihole_6} for an illustration. 
			\begin{figure}[H]
				\includegraphics[width=0.6\textwidth]{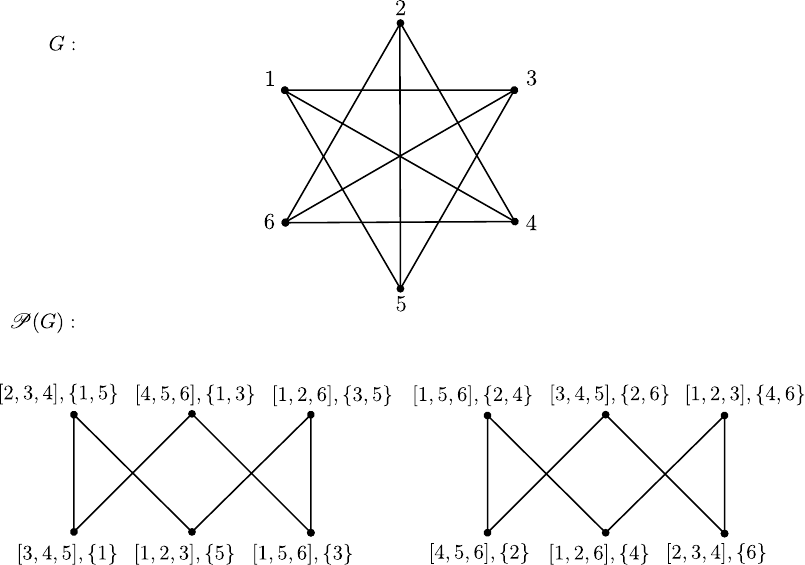} 
				\caption{The 6-antihole and its separator poset. For all antiholes it holds that the minimal separators of $\overline{C_\ell}$ are always the neighbourhoods of the vertices, thus of cardinality $\ell-3$, separating them from their neighbours in the cycle graph. The separator-posets $\mathscr{P}({\overline{C_\ell}})$ are thus complete for all $\ell$.} 
				\label{fig:antihole_6}
			\end{figure}
		\end{example} 
		\begin{example} 
			\begin{figure}[H]
				\includegraphics[width=0.6\textwidth]{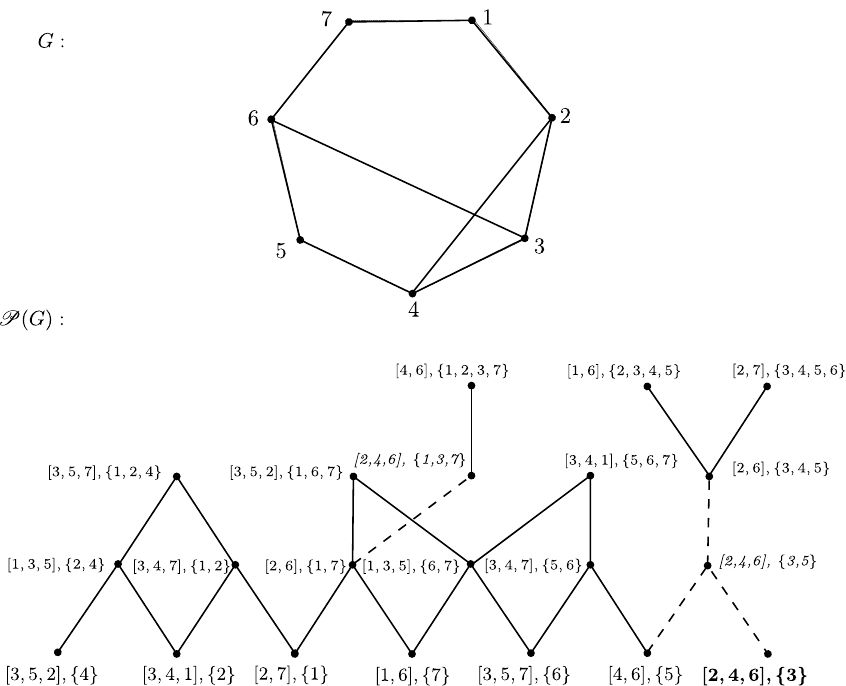} 
				\caption{The generated elements of the separator poset are italic, the added generator is bold.} 
				\label{fig:7_chord}
			\end{figure}
			Consider the graph in Figure \ref{fig:7_chord}. In this situation, there are four tuples with no descending chains, two being minimal with this property. Adding the tuple $([2,4,6], \{3\})$ to the separator poset yields a minimal generating set. Note that adding $([2,4,6], \{1,7,3\})$ or $([2,4,6], \{3,5\})$ would also work. 
		\end{example} 
		
		\subsection*{Derivation degree sequence and highest degree generator} 
		The question of determining the highest degree of a given minimal generator set was posed by Wakefield in \cite{wakefield}. 
		\begin{definition} 
			Let $G$ be a graph and let $\Theta_G$ be a minimal generating set for $D(\A)$ and write $n := \vert\Theta_G\vert$. Let $(d_1,\dots,d_n)_{\leq}$ be the ordered degree sequence of the elements in $\Theta_G$. Denote by $d := d_n$ the highest degree in the sequence. 
		\end{definition}
		We can make some general statements about this sequence using our construction: 
		
		\begin{proposition} \label{prop:subsequ} 
			Let $G$ be a graph and let $(d_1,\dots, d_n)_\leq$ be the ordered degree sequence of a minimal  generating set for $D(\A(G))$. 
			
			Then the sequence $ [i]_{0\leq i\leq \kappa(G)} + [m^{\sum_{T \in \mathscr{S}_m} (\vert\mathscr{C}(T)\vert-1)}]_{m \in \mathbb{N}}$ is a subsequence of $(d_1,\dots, d_n)_\leq$. 
		\end{proposition} 
		\begin{proof} 
			The first part of the sequence being contained in the degree sequence follows directly from the definition and corresponds to the degrees of the $\theta_i$. The second part follows from Lemmas \ref{lem:poset_gen} (every minimal separator necessarily appears in the poset) and \ref{lem:gen_compl}: we eliminate one derivation $\theta_C^T$ for each minimal separator $T$, except for one of minimal cardinality which explains the additional $\kappa(G)$ in the subsequence. 
		\end{proof} 
		\begin{corollary} 
			Let $G$ be a graph, then  $d \geq t_{\max}$. 
		\end{corollary} 
		\begin{proof} 
			This follows immediately from Proposition \ref{prop:subsequ}, since  per definition $G\setminus T$ has at least two connected components for every separator $T$. 
		\end{proof} 
		\begin{remark} 
			This is an improvement of a result in \cite{wakefield}, where the lower bound of $\text{Tri}(G)$ was found, which is defined as the maximal number of new triangles that can be made from adding an edge to $G$: 
			
			Adding an edge $\{v,w\}$ to $G$ implies that $v$ and $w$ previously were not connected through an edge, which in turn means there is a minimal $(v,w)$-separator in $G$. A vertex directly connected to both of them would have to be contained in any minimal $(v,w)$-separator and thus, its cardinality would at least be the number of triangles formed by adding the edge $\{v,w\}$. 
			
			However, this bound is not  sharp, since there could be paths of length more than 2 between $v$ and $w$ and thus the cardinality of the separator larger. 
		\end{remark} 
		\begin{proposition} 
			Let $c$ be the maximal cardinality of a clique in $G$. Then $d \geq c-1$. 
		\end{proposition}
		\begin{proof} 
			Let $\mathcal{C}$ be a clique of maximal cardinality in $G$. 
			For each minimal separator of minimal cardinality $T$, it holds that $\mathcal{C} \subseteq T \cup C$ for some $C \in \mathscr{C}(T)$. Chose any descending chain $\mathscr{D}(T,C)$ and let $i$ be the index of the first element of $\mathcal{C}$ in the chain (meaning $C_i$ does not contain any element of $\mathcal{C}$). Then $\vert T_{i+1} \vert \geq c-1$, since this component contains exactly one element of the clique and so the separator needs to contain all other $c-1$ elements of it. 
			
			What remains to show is that $T_{i+1}$ is not already generated by derivations with degree smaller than $c-1$. The separators corresponding to these derivations would have to be contained in $T_{i-1}$, thus could not contain the whole of $\mathcal{C}$, so the component would not be contained in $C$ or $\mathscr{D}{(T,C)}_{i}$ would have to be contained in the separator. In both cases, generation through Lemma \ref{lem:poset_gen} would not be possible. 
			
			Furthermore, since there is a derivation with degree at least $c-1$ for each minimal separator of minimal cardinality in $G$, not all of them are removed in the next step with Lemma \ref{lem:gen_compl}. 
		\end{proof}  
		We can also give an upper bound for $d$: 
		
		\begin{lemma} 
			Let $G$ be a graph, then $d \leq \Delta(G)$. 
		\end{lemma} 
		\begin{proof} 
			Since the generating set only uses separator-based derivations and the $\theta_i$, this follows immediately by construction: Since $\kappa(G)$ is the connectivity of $G$, it holds that $\kappa(G) \leq \delta(G)$. 
			
			Now consider an descending chain between two elements $(T,C), (T',C'), C \supset C'$ connected by an edge in the poset. Since $T, T' \in \mathscr{S}(G)$, it follows that they separate elements $a,b$ and $a', b'$ respectively. It follows that $\vert T \vert \leq \min\{\text{deg}(a), \text{deg}(b)\}$ and the analogue for $T'$. 
			
			If $C\setminus C'$ is a clique and all edges $(t,c), t \in T, c \in C\setminus C'$, we are finished. If there exist a vertex $c\in C\setminus C'$ and some vertex $c' \in T \cup C\setminus C'$ for which $\{c, c'\} \notin E(G)$, this means that there exists a minimal $(c,c')$-separator contained in $T \cup C\setminus C'$ and thus, there cannot be an edge between $(T,C), (T',C')$ in $\mathscr{P}(G)$. \qedhere

		\end{proof} 
		
		\begin{example} 
			The degree bounds are all non-strict, as we can see in Figure \ref{fig:degree}. 
			\begin{figure}[H]
				\includegraphics[width=0.45\textwidth]{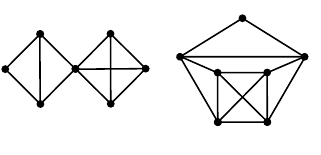} 
				\caption{Two example graphs on 7 vertices. Both of them have 4 as maximal cardinality of a clique (meaning $c-1 = 3$) and highest vertex degree 4. The left graph has $ t_{\max} = 2$ and $d = 3$. The right one has $ t_{\max} = 3$ and $d = 4$.} 
				\label{fig:degree}
			\end{figure}
		\end{example} 
	
	\subsection*{Further problems}
	As mentioned in the introduction, this project was inspired by the task of finding a graph theoretical characterization of the projective dimension of graphic arrangements. In \autocite{abe2023projective}, the case of projective dimension 1 was explored and this explicit set of generators could be used as a tool to find a general characterization. 
	
	Another line of inquiry could be to explore new concepts and definitions for hyperplane arrangements in the graphic case, especially in the case of freeness subclasses, such as MAT-free, accurate or flag-accurate arrangements (see \cite{muecksch2023flagaccurate}, \cite{mücksch}) or classes depending on generator degrees (such as strictly plus one generated arrangements \autocite{MR4276319}). 
	
	Lastly, it is an open question whether the concepts used in this paper could be generalized to larger arrangement classes to get generating sets. 
	\printbibliography
	
\end{document}